\documentclass[10pt]{amsart}
\usepackage{graphicx}
\usepackage{amssymb,amscd,amsthm,amsxtra}
\usepackage{latexsym}
\usepackage{epsfig}
\usepackage{amsmath}
\usepackage{bm}
\usepackage{amssymb}
\usepackage{amsthm}
\usepackage[usenames,dvipsnames,svgnames,table]{xcolor}
\usepackage[pagebackref,linktocpage=true,colorlinks=true,linkcolor=Blue,citecolor=BrickRed,urlcolor=RoyalBlue]{hyperref}
\usepackage[alphabetic]{amsrefs}
\usepackage[colorinlistoftodos,prependcaption,textsize=tiny]{todonotes}

\usepackage[colorinlistoftodos,prependcaption,textsize=tiny]{todonotes}
\usepackage{xcolor}

\usepackage[mathscr]{euscript} 

\usepackage{mathrsfs} 

\usepackage{pxfonts}

\usepackage[varg]{txfonts}

\numberwithin{equation}{section} 

\newtheorem{theorem}{Theorem}

\newtheorem{lemma}{Lemma}

\newcommand\dist{\mathrm{dist}}

\newcommand{\R}{\mathbb R}

\newcommand{\eps}{\varepsilon}

\newcommand{\dd}{\, \mathrm{d}}

\numberwithin{equation}{section}

\usepackage{tikz}

\newcommand{\intav}[1]{\mathchoice {\mathop{\vrule width 6pt height 3 pt depth  -2.5pt
\kern -8pt \intop}\nolimits_{\kern -6pt#1}} {\mathop{\vrule width
5pt height 3  pt depth -2.6pt \kern -6pt \intop}\nolimits_{#1}}
{\mathop{\vrule width 5pt height 3 pt depth -2.6pt \kern -6pt
\intop}\nolimits_{#1}} {\mathop{\vrule width 5pt height 3 pt depth
-2.6pt \kern -6pt \intop}\nolimits_{#1}}}

\title[On the Bernoulli problem with unbounded jumps]{\bf On the Bernoulli problem with unbounded jumps}
\author{ Stanley Snelson \& Eduardo V. Teixeira }
\address{Department of Mathematical Sciences, Florida Institute of Technology, Melbourne, FL, USA}
\email{ssnelson@fit.edu}
\address{Department of Mathematics,  University of Central Florida, Orlando, FL, USA}
\email{eduardo.teixeira@ucf.edu}

\thanks{SS was partially supported by NSF grant DMS-2213407 and a Collaboration Grant from the Simons Foundation, Award \#855061.}

\begin{document}
\maketitle

\begin{abstract}  We investigate Bernoulli free boundary problems prescribing infinite jump conditions. The mathematical set-up leads to the analysis of non-differentiable minimization problems of the form $\int  \left(\nabla u\cdot (A(x)\nabla u) + \varphi(x) 1_{\{u>0\}}\right) \dd x \to \text{min}$, where $A(x)$ is an elliptic matrix with bounded, measurable coefficients and $\varphi$ is not necessarily locally bounded. We prove universal H\"older continuity of minimizers for the one- and two-phase problems. Sharp regularity estimates along the free boundary are also obtained. Furthermore, we perform a thorough analysis of the geometry of the free boundary around a point $\xi$ of infinite jump, $\xi \in \varphi^{-1}(\infty)$. We show that it is determined by the blow-up rate of $\varphi$ near $\xi$ and we obtain an analytical description of such cusp geometries.


\tableofcontents

\end{abstract}

\section{Introduction}

For a bounded domain $\Omega \subset \R^d$, $d\ge 2$, with Lipschitz boundary, let 
\[ M_{\lambda,\Lambda}(\Omega) := \left \{A: \Omega \to \R^{d\times d}, \, A_{ij} = A_{ji}, \lambda |e|^2 \leq A_{ij}(x) e_i e_j \leq \Lambda |e|^2 \text{ for all } x\in \Omega, e \in \R^d\right \},
\]
be the space of symmetric $d\times d$ matrix-valued functions on $\Omega$ with ellipticity constants $0< \lambda \le \Lambda$. 
For fixed $A\in M_{\lambda,\Lambda}(\Omega)$ and $\varphi\colon \Omega \to \mathbb{R}$, we are interested in minimizers of the energy functional
\begin{equation}\label{AC}
	\mathcal J_{A,\varphi}(u) := \int_\Omega \left(\nabla u\cdot (A(x)\nabla u) + \varphi(x) 1_{\{u>0\}}\right) \dd x,
\end{equation}
over $H^{1}_g(\Omega)$, for some boundary condition $g$ in the trace space $H^{1/2}(\partial\Omega)$. Hereafter,  $1_{\{u>0\}}$ denotes the indicator function of the set $\{u>0 \}$. 

The energy functional described in \eqref{AC} is part of a large class of free boundary models describing cavity or jet flows, and is also related to overdetermined Bernoulli-type problems. The analysis of such free boundary problems was launched by the epoch-marking works of Alt-Caffarelli \cite{AC} and Alt-Caffarelli-Friedman \cite{ACF}, and since then has promoted major knowledge leverage across pure and applied sciences,  viz. \cite{BT,  CJK, CF, KL1, KL2, KT, KT2, CSY, CKL, DePhM, DePhSV, F, V} to cite a few. 

The main key novelty in this work is that the function $\varphi$ is only required to belong to a weak $L^q$ space, and thus it may become unbounded.  More importantly, the minimal assumption $\varphi\in L^{q}_{\rm weak}(\Omega)$ leads to multiple free boundary geometries. That is, the condition $\varphi \in L^q_{\rm weak}$ sets a {\it maximum} blow-up rate near a generic point $\xi$, viz. $|x-\xi|^{-n/q}$; however it is not granted that $\varphi(x)$ blows up at the same rate for all free boundary points. In turn, the geometry of the free boundary is modulated by the blow-up rate of $\varphi$ around a free boundary point, which may change point-by-point. A decisive new approach we discuss here concerns precise analytical quantities that allow one to classify such free boundary geometries; see Theorem \ref{t:cusps} below.

The free boundary model investigated in this article should be thought of as the dual of the Stokes conjecture, as investigated in \cite{VW}. By allowing $\varphi(x)$ to vanish at a precise rate -- in the case of the Stokes conjecture, $\varphi(x,y) = -y$ and $(0,0)$ is a free boundary point -- one can offer a variational treatment of the Stokes conjecture. {See also the series \cite{M1, M2, MN} for a related setting in which $\varphi$ vanishes in part of the domain.}  In this article, we treat the complementary case, when the Bernoulli function, $\varphi$, is allowed to become infinite. 

We further comment that we do not impose any continuity condition on the coefficients $x\mapsto A(x)$. Our regularity results are of universal nature, and hence applicable to a plethora of other models, e.g. free transmission problems, homogenization issues, etc. 

Recall that weak solutions to the homogeneous elliptic equation
\begin{equation}\label{e:homogeneous}
	 \nabla \cdot (A(x) \nabla u) = 0, \quad x\in \Omega, A \in M_{\lambda,\Lambda}(\Omega),
 \end{equation}
are locally H\"older continuous, with universal estimates. This is the content of the celebrated De Giorgi-Nash-Moser regularity theory, \cite{DeG, Moser, Nash}. That is, solutions to equation  \eqref{e:homogeneous} satisfy the estimate 
\begin{equation}\label{e:homogeneous-holder}
	\|u\|_{C^{\alpha_0}(K)} \leq C \|u\|_{L^2(\Omega)}, \quad K \Subset \Omega,
\end{equation}
for some {\it maximal} H\"older exponent $\alpha_0\in (0,1)$, depending on  $d,\lambda$, and $\Lambda$, but not on $K$ or $\Omega$. Hereafter we say a constant is \emph{universal} if it depends only on $d$, $q$, $\lambda$, and $\Lambda$. 
The constant $C>1$ in \eqref{e:homogeneous-holder} depends on universal parameters, $ K$, and  $\Omega$, but it is independent of the solution $u$.

Let us now discuss the main results proven in this article. The first key observation is that local minimizers of \eqref{AC} should satisfy an elliptic PDE like \eqref{e:homogeneous} in each phase, $\{u> 0 \}$ and $\{u< 0 \}$. Hence, any (universal) regularity estimate for local minimizers of \eqref{AC} must conform to the maximum regularity imposed by \eqref{e:homogeneous-holder}. Our first main theorem yields local regularity estimates for minimizers in the H\"older space and captures the settlement described above, in a sharp fashion.

\begin{theorem}[{\bf H\"older estimate}]\label{t:main-holder}
Let $u$ be a minimizer of $\mathcal J_{A,\varphi}(\Omega)$ over $H^1_g(\Omega)$, with $g\in H^{1/2}(\partial \Omega)$ and 
$A\in M_{\lambda,\Lambda}(\Omega)$. For any $\Omega' \subseteq \Omega$, if $\varphi \in L^q_{\rm weak}(\Omega')$ for some $q > d/2$, then for any $\delta \in (0,\alpha_0)$, $u$ is locally $\alpha$-H\"older continuous in $\Omega'$, where 
\begin{equation}\label{alpha:thm Holder}
	\alpha := \begin{cases} 1-d/(2q), & \alpha_0 > 1-d/(2q),\\ \alpha_0-\delta, &\alpha_0 \leq 1-d/(2q).\end{cases}
\end{equation}
One also has the estimate
\[ \|u\|_{C^{\alpha}(K)} \leq C \|u\|_{L^2(\Omega')}, \quad K \subset \subset \Omega'.\]
The constant $C$ depends only on $d$, $\lambda$, $\Lambda$, $q$, $\delta$, $\|\varphi\|_{L^q_{\rm weak}(\Omega')}$, and $K$.
\end{theorem}

Upon extra oscillation control of the function $x\mapsto A(x)$, higher regularity estimates for the homogeneous equation \eqref{e:homogeneous} become available, and one may take $\alpha_0 = 1$ in \eqref{alpha:thm Holder}. Choosing $\delta$ small enough, we then obtain $\alpha = 1-d/(2q)$ in Theorem \ref{t:main-holder}. In other words, if the coefficient matrix $A$ is ``continuous enough" as to allow Lipchitz estimates for the homogenous PDE \eqref{e:homogeneous}, then minimizers display the sharp H\"older regularity with exponent $1-d/(2q)$. See \cite{BM} for optimal conditions yielding Lipschitz regularity of solutions in great generality. 

We highlight that Theorem \ref{t:main-holder} is for the two-phase problem, i.e., it does not carry any sign restriction on $u$. 
Our next main result gives improved regularity for the one-phase problem. It says that {\it at free boundary points}, the sharp H\"older exponent $1-d/(2q)$ is achieved, regardless of the value of $\alpha_0$ in \eqref{e:homogeneous-holder}. 

\begin{theorem}[{\bf Improved regularity at free boundary points}]\label{t:improved}
With $u$, $g$, and $\varphi$ as in Theorem \ref{t:main-holder}, suppose the boundary data $g$, and therefore $u$, are nonnegative. Let $x_0$ be a boundary point of $\{u>0\}$, and suppose $\varphi \in L^q_{\rm weak}(\Omega')$ for some open $\Omega'\subseteq\Omega$ containing $x_0$.  Then $u \in C^{1-d/(2q)}$ at $x_0$, in the sense that 
$$
	\left | u(x) \right | \le C |x-x_0|^{1-d/(2q)} {\|u\|_{L^2(\Omega)}},
$$
for a constant $C>0$ that depends only on universal parameters, $\Omega'$, and $\Omega$.
\end{theorem}

We also obtain nondegeneracy estimates that say $u$ must grow at least at a certain rate near free boundary points. The asymptotics of these lower bounds are determined by the local blowup behavior of $\varphi$---more specifically, if $\varphi$ satisfies an average lower bound like
\begin{equation*}
|\{\varphi >t\}\cap B_r(x_0)| \geq \min\left( c t^{-p}, |B_r(x_0)|\right),
\end{equation*}
 for some $c>0$, all $t>0$, and sufficiently small $r>0$, then $u$ satisfies
\begin{equation}\label{e:nondeg}
|u(x)|\geq c|x-x_0|^{1-d/(2p)},
\end{equation}
for $x$ near $x_0$. See Lemma \ref{l:nondegeneracy} for a more precise (and more general) statement. In particular, this lower bound implies $|\nabla u|\to \infty$ as $x\to x_0$ from inside $\{u>0\}$. Note that for $\varphi \in L^q_{\rm weak}$, the exponent $p\le q$ may be strictly smaller than $q$, and this causes extra subtleties in  our analysis. 

Next, we address the geometry of the free boundary $\partial \{u>0\}$. The key point is to describe how the rate at which $\varphi$ blows up near a free boundary point impacts the free boundary configuration.

\begin{theorem}[{\bf Control on the severity of cusps in the free boundary}]\label{t:cusps}
Let $u$ be a minimizer of $\mathcal J_{A,\varphi}$ over $H^1_g(\Omega)$ as above, with $g\geq 0$. 
Let $x_0 \in \partial \{u>0\}$, and assume that for some $r_0>0$, $q_+> d/2$, and $q_-\leq q_+$, 
\begin{equation}\label{e:q-+}
\begin{split} 
&   \varphi \in L^{q_+}(B_{r_0}(x_0)), \quad \text{ and } \\
 &   |\{\varphi(x) > t\} \cap B_r(x_0)| \geq \min\left( c_0 t^{-q_-}, |B_r(x_0)|\right) \quad \text{ for all } r\in (0,r_0), t>0.
\end{split}
\end{equation}
Then, with $\alpha_1 = 1-d/(2q_+)$, $\alpha_2 = 1-d/(2q_-)$, there exists a constant $c>0$ such that for all $r\in (0,r_0)$,
\[ 
\frac {|\{u>0\} \cap B_r(x_0)|} { |B_r(x_0)|^{\alpha_2/\alpha_1}} \geq c, \quad \text{and} \quad \frac {|\{u=0\} \cap B_r(x_0)|} { |B_r(x_0)|^{P(q_-,q_+)}} \geq c,
\]
where 
\[
P(q_-,q_+) = \left( \frac {\alpha_2}{\alpha_1} - \frac 1 {q_-}\right) \frac{q_+}{q_+-1}.
\]
In particular, if $\varphi$ is such that $q_- = q_+$ in \eqref{e:q-+}, 
then there exists $c\in (0,1)$ such that
\[ c \leq \frac {|\{u>0\} \cap B_r(x_0)|} { |B_r(x_0)|} \leq 1-c,\]
for all $r\in (0,r_0)$.
\end{theorem}

{On the other hand, nondegeneracy estimates (see \eqref{e:nondeg} or Lemma \ref{l:nondegeneracy}) imply some limits on the regularity of $\partial\{u>0\}$ via classical potential theory. Since $u$ solves \eqref{e:homogeneous} in $\{u>0\}$ and is not Lipschitz at a boundary point $x_0$ where \eqref{e:nondeg} holds, we conclude $\{u>0\}$ does not satisfy an exterior sphere condition and in particular cannot be $C^2$ at $x_0$. Even further, the exponent $1-d/(2p)$ in \eqref{e:nondeg} implies some upper limit (depending on $d$, $p$, $\lambda$, and $\Lambda$) on the aperture of any exterior cone at $x_0$. 
}

 The rest of the paper is organized as follows. In Section \ref{sct pre} we discuss some preliminary results needed for the proofs of the main Theorems. In particular we discuss the scaling feature of the problem and establish a Caccioppoli-type estimate. In Section \ref{sct Holder} we prove Theorem \ref{t:main-holder}, by means of a careful approximation analysis. In the intermediary Section \ref{sct borderline} we discuss the limiting case when $\varphi \in L^{d/2}_{\text{weak}}$, and obtain universal BMO estimates. In Section \ref{s:improved} we establish the sharp $C^{1-d/(2q)}$ regularity at one-phase free boundary points. The interesting feature here is that such an estimate is not limited by the universal regularity theory of $A$-harmonic functions. Section \ref{sct geometry} is devoted to the geometric analysis of the free boundary, where we prove Theorem \ref{t:cusps}. Finally, in Section \ref{sct example} we discuss an example showing the free boundary may indeed intersect the infinite points of the Bernoulli function $\varphi$ in a non-trivial subregion of the domain.

\section{Preliminaries}\label{sct pre}

We begin by recalling that for $q\geq 1$, a function $f$ lies in the space $L^q_{\rm weak}(\Omega)$ if 
\[ 
	\|f\|_{L^q_{\rm weak}(\Omega)} := \sup_{t>0} \, t \, |\{x\in \Omega : f(x) > t\}|^{1/q} < \infty.
\]

Note that existence of minimizers for $\mathcal J_{A,\varphi}$ over $H^1_g(\Omega)$ follows as in \cite{AC}, so we omit the details.

Next, let us investigate how the minimization problem for $\mathcal J_{A,\varphi}$ transforms under translation and rescaling around a point. For $\gamma\in \R$, define the more general functional 
\[ 
\mathcal J_{A,\varphi,\gamma}(u) := \int_{\Omega} \left( \nabla u\cdot (A(x)\nabla u) + \varphi(x) 1_{\{u>\gamma\}}\right)\dd x.
\]
Clearly, $\mathcal J_{A,\varphi} = \mathcal J_{A,\varphi,0}$. We suppress the dependence of $\mathcal J_{A,\varphi,\gamma}$ on the domain $\Omega$, which will always be clear from context.
\begin{lemma}\label{l:local-min}
For some domain $\Omega$, $\varphi\in L^q_{\rm weak}(\Omega)$, and boundary data $g \in H^{1/2}(\partial \Omega)$, let $u$ be a minimizer of $\mathcal J_{A,\varphi}$ over $H^1_g(\Omega)$. Then, for any $x_0\in \Omega$, $\kappa \geq 0$, $0<r<\mathrm{dist}(x_0,\partial \Omega)$, and $\gamma\in \R$, 
the function 
\[
v(x):= \kappa (u(x_0+rx)- \gamma)
\] 
is a minimizer of $\mathcal J_{\tilde A,\tilde \varphi, \tilde \gamma}$ over $H^1_{\tilde g}(B_1)$, with 
\[ 
\begin{split}
\tilde \varphi(x) &= \kappa^2 r^2 \varphi(x_0+rx),\\
\tilde \gamma &= -\kappa \gamma,\\
\tilde A(x) &= A(x_0+rx),\\
\tilde g(x) &= \kappa (u(x_0 + rx)-\gamma), \, \,\, x\in \partial B_1.
\end{split}
\]
\end{lemma}
\begin{proof}
Direct calculation.
\end{proof}

Next, by a standard argument, we can show that minimizers of $\mathcal J_{A,\varphi,\gamma}$ are subsolutions of the homogeneous equation \eqref{e:homogeneous}:
\begin{lemma}\label{l:subsolution}
Let $u$ be a minimizer of $\mathcal J_{A,\varphi,\gamma}$ over 
\[ H^1_{g}(\Omega) :=  \{w \in H^1(\Omega), w = g \text{ on } \partial \Omega \text{ in trace sense}\}, \]
for some $g\in H^{1/2}(\partial\Omega)$. Then 
\[ \int_{\Omega} (A(x)\nabla u) \cdot \nabla v \geq 0, \quad  v\in C_0^\infty(\Omega), v\geq 0.\]
\end{lemma}
\begin{proof}
This lemma follows by noting that $\mathcal J_{A,\varphi,\gamma}(u) \leq \mathcal J_{A,\varphi,\gamma}(u-\eps v)$ and $1_{\{u-\eps v>\gamma\}} \leq 1_{\{u>\gamma\}}$ for all $\eps>0$ and all nonnegative $v\in C_0^\infty(\Omega)$.
\end{proof}
{ From Lemma \ref{l:subsolution} and the maximum principle, we conclude minimizers are bounded whenever $g\in L^\infty(\partial\Omega)$, with the estimate $\|u\|_{L^\infty(\Omega)} \leq \|g\|_{L^\infty(\partial\Omega)}$. }

Our next lemma is a Caccioppoli-type estimate that will be needed in our compactness argument:
\begin{lemma}\label{l:caccioppoli}
There exists a constant $C>0$ depending on $d$, $q$, $\lambda$, and $\Lambda$, such that any minimizer $u$ of $\mathcal J_{A,\varphi,\gamma}$ over $H^1_g(B_1)$ satisfies
\begin{equation*}
\int_{B_{1/2}} |\nabla u|^2 \dd x \leq  C \left(\int_{B_1} |u|^2 \dd x +  \|\varphi\|_{L^q_{\rm weak}(B_1)}\right).
\end{equation*}
\end{lemma}
Note that $u$ is allowed to change sign in this lemma. If we were concerned only with the one-phase problem, $u$ would be a nonnegative subsolution of \eqref{e:homogeneous}, 
and we could apply an existing Caccioppoli estimate such as \cite[Lemma 3.27]{potential-theory-book}.
\begin{proof}
For any $\zeta \in C_0^\infty(B_1)$ with $0\leq \zeta \leq 1$, we 
let $w = u(1-\zeta^2)$ so that $w = u$ on $\partial B_1$. The minimizing property $\mathcal J_{A,\varphi,\gamma}(u) \leq \mathcal J_{A,\varphi,\gamma}(w)$ implies
\begin{equation}\label{e:compare}
\begin{split}
\int_{B_1} \nabla u\cdot(A(x) \nabla u)\dd x &\leq \int_{B_1} \left[\nabla w\cdot (A(x)\nabla w) + \varphi(x)\left(1_{\{w>\gamma\}} - 1_{\{u>\gamma\}}\right)\right]\dd x.
 \end{split}
 \end{equation}
With $w = u_n(1-\zeta^2)$, straightforward calculations imply
\[\begin{split}
 \int_{B_1} \nabla u\cdot (A(x)\nabla u \zeta^2 (2-\zeta^2)\dd x &\leq 4 \int_{B_1} \left(u^2 \zeta^2 \nabla \zeta\cdot (A(x) \nabla \zeta) - u \zeta (1-\zeta^2)\nabla \zeta\cdot (A(x) \nabla u) \right) \dd x\\
 &\quad + C_{d,q} \|\varphi\|_{L^q_{\rm weak}(B_1)}.
 \end{split}\]
In the last term on the right, we used H\"older's inequality together with the fact that $\|\varphi\|_{L^r(B_1)} \lesssim \|\varphi\|_{L^q_{\rm weak}(B_1)}$ for $1< r< q$. With $0\leq 1-\zeta^2\leq 1$ and Young's inequality, we have
\[\begin{split}
 \lambda \int_{B_1} \zeta^2 |\nabla u|^2 \dd x &\leq 4 \Lambda \int_{B_1} \left( u^2 |\nabla \zeta|^2 + |u| \zeta |\nabla \zeta| |\nabla u|\right)\dd x + + C_{d,q} \|\varphi\|_{L^q_{\rm weak}(B_1)}\\
&\leq \frac \lambda 2 \int_{B_1} \zeta^2 |\nabla u|^2 \dd x + \frac {2\Lambda}{\lambda} \int_{B_1} u^2 |\nabla \zeta|^2  \dd x +  C_{d,q} \|\varphi\|_{L^q_{\rm weak}(B_1)}.
\end{split}
  \]
Choosing $\zeta$ equal to $1$ in $B_{1/2}$ and 0 outside $B_{3/4}$, with $|\nabla \zeta|$ bounded by a constant, the proof is complete.
\end{proof}

\section{H\"older continuity}\label{sct Holder}

In this section, we establish a universal H\"older estimate for the one- and two-phase problems. In Section \ref{s:improved}, we will further improve such an estimate (at the free boundary) to the sharp exponent $1-d/(2q)$, regardless of the regularity of $A$-harmonic functions.

The analysis will be based on the following key approximation lemma, which says minimizers of $\mathcal J_{A,\varphi,\gamma}$ are close to minimizers of $\int_\Omega \nabla u\cdot (A(x)\nabla u) \dd x$, if the norm of $\varphi$ is sufficiently small in $L^q_{\rm weak}$, c.f. \cite{Teix, Teix2, Teix3} for related analysis employed in  ``non-free boundary" PDE models.

\begin{lemma}\label{l:rho}
Given $\tau>0$, there exists $\eps = \eps(d, q,\lambda,\Lambda,\tau)>0$ such that for any $\varphi\in L^q_{\rm weak}(B_1)$ with $q\in (1,\infty)$ and $\|\varphi\|_{L^q_{\rm weak}(B_1)}\leq \eps$, any $A\in M_{\lambda,\Lambda}(B_1)$, any $g \in H^{1/2}(\partial B_1)$, any $\gamma\in\R$, and any minimizer $u$ of $\mathcal J_{A,\varphi,\gamma}$ over $H^1_{g}(B_1)$ such that $\fint_{B_1} u^2 \dd x \leq 1$, there holds 
\[   \int_{B_{1/2}} |u - h|^2 \dd x \leq \tau,\]
where $h\in H^1(B_{1/2})$ satisfies $\fint_{B_{1/2}} h^2 \dd x \leq 2^{-d}$ and is a weak solution to 
\begin{equation}\label{e:homogeneous1}
 \nabla \cdot (A_0(x) \nabla h) = 0, \quad x\in B_{1/2},
 \end{equation}
for some $A_{0}\in M_{\lambda,\Lambda}(B_{1/2})$.
\end{lemma}

In \eqref{e:homogeneous1}, by \emph{weak solution} we mean that $\int_{B_{1/2}} \nabla w \cdot (A(x)\nabla h) \dd x = 0$ for all $w\in H^1_0(B_{1/2})$.

\begin{proof}
Assume there is no $\eps$ satisfying the conclusion. Then there exists a sequence $\eps_n \to 0$ and corresponding sequences $g_n \in H^{1/2}(B_1)$, $\varphi_n \in L^q_{\rm weak}(B_1)$ with $\|\varphi_n\|_{L^q_{\rm weak}(B_1)} = \eps_n$, $A_n\in M_{\lambda, \Lambda}(B_1)$, $\gamma_n\in \R$, and $u_n$ minimizing $\mathcal J_{A_n,\varphi_n,\gamma_n}$ over $H^1_{g_n}(B_1)$ with $\fint_{B_1} u_n^2 \dd x \leq 1$.

By the contradiction hypothesis, we have
\begin{equation}\label{e:contradiction}
\int_{B_1} |u_n - h|^2 \dd x > \tau,
\end{equation}
for all $n$ and all $h$ as in the statement of the lemma.

By the Caccioppoli estimate Lemma \ref{l:caccioppoli} and our assumption that $\fint_{B_1} u_n^2 \dd x \leq 1$, we see that 
the sequence $u_n$ is bounded in $H^1(B_{1/2})$, and a subsequence (still denoted $u_n$) converges weakly in $H^1(B_{1/2})$ and strongly in $L^2(B_{1/2})$ to a limit $u_0$. 
Since $A_n \in M_{\lambda,\Lambda}(B_1)$ for all $n$, by passing to a further subsequence, we can ensure $A_n$ converges in the weak $L^2(B_{1/2}, \R^{d\times d})$ sense to some $A_0\in M_{\lambda, \Lambda}(B_{1/2})$.

We want to show that $u_0$ solves \eqref{e:homogeneous1} in $B_{1/2}$. At this stage, we do not have much information about regularity of $u_n$, and therefore it seems difficult to work with the Euler-Lagrange equation for $u_n$ (particularly the free boundary condition) directly. That is why we use variational techniques to characterize the limit of $u_n$.

For any $w\in C_0^\infty(B_{1/2})$, we have $\mathcal J_{A_n,\varphi_n,\gamma_n}(u_n) \leq \mathcal J_{A_n,\varphi_n,\gamma_n}(u_n+w)$, or
\begin{equation}\label{e:non-negative}
\begin{split} 
0 &\leq \int_{B_{1/2}}  \left[\nabla (u_n+w)\cdot (A_n(x)\nabla (u_n+ w)) - \nabla u_n\cdot(A_n(x)\nabla u_n)\right] \dd x   \\
&+ \int_{B_{1/2}} \varphi_n(x) \left( 1_{\{u_n+w>\gamma\}} - 1_{\{u_n>\gamma\}}\right) \dd x\\
&= : I_1 + I_2.
\end{split}
\end{equation}
To pass to the limit as $n\to \infty$, we first note that $\|\varphi_n\|_{L^r(B_1)} \lesssim \|\varphi_n\|_{L^q_{\rm weak}(B_1)}$ for any $r\in (1,q)$, and therefore $|I_2|\leq 2|B_{1/2}|^{(r-1)/r} \|\varphi_n\|_{L^r(B_{1/2})} \to 0$.  

Next, we consider $I_1$. Note that
\begin{equation}\label{e:I1}
\begin{split}
 I_1 &= \int_{B_{1/2}} [ 2 \nabla w\cdot (A_n(x) \nabla u_n) + \nabla w \cdot (A_n(x) \nabla w)] \dd x.
\end{split}
\end{equation}
The second integral in \eqref{e:I1} converges to $\int_{B_{1/2}} \nabla w \cdot (A_0(x)\nabla w)\dd x$, by the weak convergence of $A_n\rightharpoonup A_0$. We claim the first integral in \eqref{e:I1} converges to $\int_{B_{1/2}} 2\nabla w \cdot (A_0(x)\nabla u_0)\dd x$. Indeed, 
\[\begin{split}
\int_{B_{1/2}} \nabla w\cdot (A_n(x)\nabla u_n) \dd x - \int_{B_{1/2}} \nabla w\cdot (A_0(x)\nabla u_0) \dd x  &= \int_{B_{1/2}} \nabla w \cdot (A_n(x)\nabla (u_n-u_0)) \dd x\\
&\quad + \int_{B_{1/2}} \nabla w \cdot( (A_n(x) - A_0(x)) \nabla u_0) \dd x.
\end{split}\]
The last integral converges to zero by the weak convergence of $A_n$. For the first integral on the right, we have
\[ \int_{B_{1/2}} \nabla w\cdot (A_n(x)\nabla (u_n - u_0)) \dd x \leq \Lambda \int_{B_{1/2}} |\nabla w| \cdot |\nabla (u_n-u_0)| \dd x,\]
and the weak convergence of $\nabla u_n \rightharpoonup \nabla u_0$ implies the last expression converges to 0.

Returning to \eqref{e:non-negative}, we take the limit as $n\to \infty$ to obtain
\[ 0 \leq \int_{B_{1/2}} [\nabla (u_0 + w)\cdot (A_0(x)\nabla (u_0+w)) - \nabla u_0\cdot (A_0(x) \nabla u_0)] \dd x.\]
This implies $\int_{B_{1/2}} \nabla u_0 \cdot (A_0 \nabla u_0) \dd x \leq \int_{B_{1/2}} \nabla w \cdot (A_0\nabla w)\dd x$ for any $w\in C_0^\infty$, and therefore $u_0$ solves the homogeneous equation \eqref{e:homogeneous1} in $B_{1/2}$ in the sense of distributions. By a standard argument, $u_0$ is also a weak solution to \eqref{e:homogeneous1}.

For $n$ sufficiently large, we have
\[ \int_{B_{1/2}} |u_n-u_0|^p \dd x \leq \tau,\]
contradicting \eqref{e:contradiction}.
\end{proof}

Next, we show that when $\|\varphi\|_{L^q_{\rm weak}}$ is small, since minimizers $u$ are close to solutions $h$ of the homogeneous equation \eqref{e:homogeneous1}, the H\"older regularity of $h$ implies some local integrability estimates for $u$:

\begin{lemma}\label{l:local-holder}
Let $\alpha_0 = \alpha_0(d,\lambda,\Lambda)$ be the exponent from \eqref{e:homogeneous-holder}. For any $\alpha \in (0,\alpha_0)$, there exist $\eps>0$, $r_0\in (0,\frac 1 4)$, $K>0$ and $C_0>0$, depending only on universal quantities and $\alpha$, such that for any minimizer $u$ of $\mathcal J_{A,\varphi,\gamma}$ over $H^1_{g}(B_1)$ with $\|\varphi\|_{L^{q}_{\rm weak}(B_1)} \leq \eps$, $\gamma \in\R$, $A\in M_{\lambda,\Lambda}(B_1)$, and $\fint_{B_1} u^2 \dd x \leq 1$, there holds
\begin{equation}\label{e:mu-bound}
 \fint_{B_{r_0}} |u - \mu|^2 \dd x \leq  r_0^{2\alpha},
 \end{equation}
for some constant $\mu$ with $|\mu|\leq K$.  
\end{lemma}

\begin{proof}
Let $\tau>0$ be a constant to be chosen later. With $u$ as in the statement of the lemma, let $h:B \to \R$ be the solution to \eqref{e:homogeneous1} in $B_{1/2}$ with
\[ \int_{B_{1/2}} |u - h|^2 \dd x < \tau,\]
whose existence is granted by Lemma \ref{l:rho}. The choice of $\tau$ will determine $\eps$. Since $\fint_{B_{1/2}} h^2 \dd x \leq 2^{-d}$, the interior regularity theory of the equation satisfied by $h$ (e.g. \cite[Theorem 8.24]{gilbargtrudinger}) implies there is a $C_0>0$ with 
\[  |h(x) - h(0)|  \leq C_0 |x|^{\alpha_0},\]
for any $x\in B_{1/4}$. For any $r_0 \in (0,\frac 1 4)$, we then have
\begin{equation}\label{e:triangle}
\begin{split}
\fint_{B_{r_0}} |u(x) - h(0)|^2 \dd x &\leq 2\left( \fint_{B_{r_0}} |u(x) - h(x)|^2 \dd x + \fint_{B_{r_0}} |h(x) - h(0)|^2 \dd x\right)\\
 &\leq  2 \omega_d^{-1} r_0^{-d}\tau + 2C_0 r_0^{2\alpha_0}.
\end{split}
\end{equation}
where $\omega_d$ is the volume of the $d$-dimensional unit ball. Choosing
\[ r_0 = \left(\frac 1 {4C_0} \right)^{1/ [2(\alpha_0-\alpha)]}, \quad \tau = \frac 1 {4} \omega_d r_0^{d+2\alpha},\]
we have
\[\fint_{B_{r_0}} |u(x) - h(0)|^2 \dd x \leq  r_0^{2\alpha}.\]
The estimate \eqref{e:mu-bound} now follows by choosing $\mu = h(0)$. The bound $|\mu|\leq K$ is a result of the interior $L^2$-to-$L^\infty$ estimate satisfied by $h$ \cite[Theorem 8.17]{gilbargtrudinger}. 
\end{proof}

Next, by iterating Lemma \ref{l:local-holder}, we show that under similar conditions, minimizers are H\"older continuous at the origin. 

\begin{lemma}\label{l:global-holder}
For $\frac d 2 < q < \infty$, fix a small $\delta \in (0,\alpha_0)$ and let
\[\alpha :=  \begin{cases} 1-\frac d {2q}, & \alpha_0 \geq 1- \frac d {2q},\\
\alpha_0-\delta, &\alpha_0 < 1-\frac d {2q},\end{cases}\]
where $\alpha_0$ is as in Lemma \ref{l:local-holder}. Then there exist $\eps, C>0$ and $r_0 \in (0,\frac 1 4)$, depending on universal quantities and $\delta$, 
such that for any minimizer $u$ of $\mathcal J_{A,\varphi}$ over $H^1_g(B_1)$ with $\|\varphi\|_{L^{q}_{\rm weak}(B_1)} \leq \eps$ and $\fint_{B_1} u^2 \dd x \leq 1$, there holds
\[ |u(x)|\leq C \quad \text{ and } \quad  |u(x)-u(0)| \leq C|x|^\alpha, \quad \mbox{if } |x| < r_0.\]
\end{lemma}

\begin{proof}
With $\alpha$ as in the statement of the lemma, let $\eps$ and $r_0$ be the corresponding constants from Lemma \ref{l:local-holder}, and let $u$ be a minimizer of $\mathcal J_{A,\varphi}$ over $H_g^1(B_1)$. Our goal is to show by induction that
\begin{equation}\label{e:induction}
 \fint_{B_{r_0^k}} (u- \mu_k)^2 \dd x \leq r_0^{2k\alpha}, \quad k = 1,2, \ldots,
\end{equation}
for some convergent sequence $\mu_k$.

The base case $k=1$ follows directly from Lemma \ref{l:local-holder}, for some constant $\mu_1$ with $|\mu_1| \leq K$. Here, $K$ depends only on universal quantities and $\delta$. 
Now, assume \eqref{e:induction} holds for some $k\geq 1$ and some constant $\mu_k$. Define
\[ v(x) =  \frac {u(r_0^k x) - \mu_k }{r_0^{k\alpha}}, \quad x\in B_1.\] 
By Lemma \ref{l:local-min}, $v$ is a minimizer of $\mathcal J_{A_k,\varphi_k,\gamma_k}$ over $H^1_{g_k}(B_1)$, where 
\[ \begin{split}
\varphi_k(x) &= r_0^{2k(1-\alpha)} \varphi(r_0^k x),\\
\gamma_k &= -r_0^{k\alpha} \mu_k,\\
A_k(x) &= A(r_0^kx),\\ 
g_k(x) &= r_0^{-k\alpha} (u(r_0^k x) - \mu_k), \quad x\in \partial B_1.
\end{split}\]
The inductive hypothesis \eqref{e:induction} implies 
\[\fint_{B_1}v^2 \dd x = r_0^{-2k\alpha} \fint_{B_{r_0^k}} (u-\mu_k)^2 \dd x \leq 1.\]
Our choice of $\alpha$ implies $2-2\alpha - d/q \geq 0$, and
\[\|\varphi_k\|_{L^q_{\rm weak}(B_1)} = r_0^{k(2-2\alpha - d/q)}\|\varphi\|_{L^q_{\rm weak}(B_{r_0^k})} \leq \|\varphi\|_{L^q_{\rm weak}(B_1)} < \eps.\]
Therefore, we can apply Lemma \ref{l:local-holder} to conclude 
\[ \fint_{B_{r_0}} |v - \mu|^2 \dd x \leq  r_0^{2\alpha},\]
for some constant $\mu$ with $|\mu|\leq K$. Translating back to $u$ with the change of variables $x \mapsto r_0^k x$, the previous inequality becomes
\[ \fint_{B_{r_0^{k+1}}} |u -\mu_k - r_0^{k\alpha} \mu|^2 \dd x \leq  r_0^{2\alpha(k+1)}.\]
We have established \eqref{e:induction} with $\mu_{k+1} = \mu_k + r_0^{k\alpha} \mu$. To show that the sequence $\{\mu_k\}$ is convergent, note that $|\mu_{k+1} - \mu_k| = |r_0^{k\alpha}\mu|\leq K r_0^{k\alpha}$, and for any $j>k$,
\begin{equation}\label{e:cauchy}
 |\mu_j - \mu_k| \leq K  (1-r_0^\alpha)^{-1} r_0^{k\alpha}, 
 \end{equation} 
which demonstrates that $\{\mu_k\}$ is Cauchy. In fact, letting $\mu_0 = \lim_{k\to \infty} \mu_k$, taking $j\to \infty$ in \eqref{e:cauchy} yields $|\mu_k - \mu_0| \leq K (1-r_0^\alpha)^{-1} r_0^{k\alpha}$ for all $k$.

Finally, for $0<r\leq r_0$, choose $k$ such that $r_0^{k+1} < r \leq r_0^k$. Then, by \eqref{e:induction},
\[ \fint_{B_r} |u - \mu_0|^2 \dd x \leq  2 \fint_{B_r} |u-\mu_k|^2 \dd x + 2 |\mu_k-\mu_0|^2 \leq r_0^{-2\alpha}\left(\frac {2} {r_0^{d}} + 2 K^2(1-r_0^\alpha)^{-2}\right) r^{2\alpha}, \]
which implies the H\"older estimate $|u(x) - u(0)|\leq C|x|^\alpha$, with $C$ as in the statement of the theorem. Since $|u(0)| = |\mu_0|\leq K(1-r_0^\alpha)^{-1}+K$, the triangle inequality implies $|u(x)|\leq C$ in $B_{r_0}$, with $C$ as above.
\end{proof}

We are now ready to prove the main result of this section:

\begin{theorem}\label{t:holder-prelim}
Let $u$ be a minimizer of $\mathcal J_{A,\varphi}$ over $H^1_g(\Omega)$, with $g\in H^{1/2}(\partial\Omega)$. For any $\Omega'\subset\subset \Omega$, if $\varphi \in L^q_{\rm weak}(\Omega')$ with $d/2< q < \infty$, then $u$ is H\"older continuous in $\Omega'$, with 
\[ \|u\|_{C^\alpha(\Omega')} \leq C \|u\|_{L^2(\Omega)},\]
where $\alpha = \min\{1-d/(2q), \alpha_0^{\color{blue}-}\}$ and $C$ depend only on universal quantities, $\Omega'$, $\Omega$, and $\|\varphi\|_{L^q_{\rm weak}(\Omega)}$.
\end{theorem}

\begin{proof}
Let $u$ be as in the statement, and let $x_0 \in \Omega$ be arbitrary. To recenter around $x_0$ and ensure the hypotheses of Lemma \ref{l:global-holder} are satisfied, we define
\[w(x) := \kappa u(x_0+rx), \quad x\in B_1,\]
where 
\[ r = \min\left\{\frac 1 2 \mathrm{dist}(x_0,\partial \Omega), \left(\frac \eps {\|\varphi\|_{L^q_{\rm weak}(\Omega)}}\right)^{q/(2q-d)}\right\}, \quad \kappa =  \left( \frac {1} {\fint_{B_r(x_0)} u^2 \dd x}\right)^{1/2} , \]
and $\eps$ is the constant from Lemma \ref{l:local-holder}. Then $\fint_{B_1} w^p \dd x \leq 1$, and by Lemma \ref{l:local-min}, $w$ minimizes $\mathcal J_{\tilde A,\tilde \varphi}$ over $H_{\tilde g}^1(B_1)$, with $\tilde A(x) = A(x_0+rx)$, $\tilde \varphi(x) = r^2 \kappa^2 \varphi(x_0+rx)$, and $\tilde g(x) = \kappa g(x_0 + rx)$ for $x\in \partial B_1$. Since $\|\tilde\varphi\|_{L^q_{\rm weak}(B_1)} \leq \kappa^2 r^{2-d/q}\|\varphi\|_{L^q_{\rm weak}(\Omega)} < \eps$, we may apply Lemma \ref{l:global-holder} and obtain $|w(x)-w(0)| \leq C|x|^\alpha$ for $|x|\leq r_0$. Since $C$ and $r_0$ depend only on $d$, $q$, $\lambda$, and $\Lambda$, when we translate back to $u$, we conclude
\[ |u(x) - u(x_0)| \leq C\|u\|_{L^2(\Omega)}|x-x_0|^\alpha, \quad |x-x_0| \leq r',\]
with $C$ and $r'$ depending only on $d$, $q$, $\lambda$, $\Lambda$, and $\mbox{dist}(x_0, \partial \Omega)$. The $L^\infty$ norm of $u$ is also bounded uniformly in any compact subset of $\Omega$, by Lemma \ref{l:global-holder}.
\end{proof}

\section{The borderline case $q = d/2$}\label{sct borderline}

In the limiting case $q \searrow d/2$, the corresponding exponent $\alpha $ the iteration of Lemma \ref{l:global-holder} becomes zero. As usual, one should not expect to obtain continuity. In this section we obtain a (sharp) estimate in the space of bounded mean oscillation (BMO) functions. Recall that a function $f$ on $\Omega$ is BMO if 
\[ \|f\|_{BMO(\Omega)} := \sup\left\{ N : \fint_{B} |f-f_B| \dd x \leq N \, \text{ for every ball } B\subset \Omega\right\} < \infty,\]
where $f_B$ is the average of $f$ over the ball $B$. 

First, we need a corresponding version of Lemma \ref{l:local-holder} for a BMO-type estimate:

\begin{lemma}\label{l:local-BMO}
There exist $\eps>0$ and $r_0\in (0,\frac 1 4)$, depending only on $d$, $\lambda$, and $\Lambda$, such that for any minimizer $u$ of $\mathcal J_{A,\varphi,\gamma}$ over $H^1_{g}(B_1)$ with $\|\varphi\|_{L^{d/2}_{\rm weak}(B_1)} \leq \eps$, $\gamma \in \R$, and $\fint_{B_1} u^2 \dd x \leq 1$, there holds
\begin{equation}
 \fint_{B_{r_0}} |u - u_{r_0}|^2 \dd x \leq 1, 
 \end{equation}
where $u_{r_0} = \fint_{B_{r_0}} u \dd x$.
\end{lemma}

\begin{proof}
First, we recall a general inequality: for any $\mu \in \R$ and $r_0>0$,
\begin{equation}\label{e:average}
 \fint_{B_{r_0}} | u - u_{r_0}|^2 \dd x \leq 4 \fint_{B_{r_0}} |u-\mu|^2 \dd x.
\end{equation}
{This inequality follows by writing
\[
\begin{split}
\left(\fint_{B_{r_0}} \left| u - \fint_{B_{r_0}} u \dd x\right|^2\dd x\right)^{1/2} &\leq \left( \fint_{B_{r_0}} |u- \mu|^2 \dd x \right)^{1/2} + \left(\fint_{B_{r_0}} |\mu - \fint_{B_{r_0}} u\dd x|^2 \dd x\right)^{1/2} \\
&\leq  \left( \fint_{B_{r_0}} |u- \mu|^2 \dd x \right)^{1/2} + \fint_{B_{r_0}} |u-\mu| \dd x\\
&\leq 2\left( \fint_{B_{r_0}} |u- \mu|^2 \dd x \right)^{1/2}.
\end{split}
\]
}

Now, let $\tau>0$ be a constant to be chosen later. As in the proof of Lemma \ref{l:local-holder}, let $h:B_{1/2} \to \R$ be the solution of the homogeneous equation \eqref{e:homogeneous1} with
\[ \int_{B_{1/2}} |u - h|^2 \dd x < \tau,\]
given by Lemma \ref{l:rho}. As above, we have 
\[  |h(x) - h(0)|  \leq C |x|,\]
for any $x\in B_{1/4}$. For $r_0 \in (0,\frac 1 4)$, we then have
\begin{equation}\label{e:triangle2}
\begin{split}
\fint_{B_{r_0}} |u(x) - h(0)|^2 \dd x &\leq 2\left( \fint_{B_{r_0}} |u(x) - h(x)|^2 \dd x + \fint_{B_{r_0}} |h(x) - h(0)|^2 \dd x\right)\\
 &\leq  2 \omega_d^{-1} r_0^{-d}\tau + 2C r_0^{2}.
\end{split}
\end{equation}
Choosing
\[ r_0 = \frac 1 4 \left(\frac {1} {C} \right)^{1/2}, \quad \tau = \frac {\omega_d r_0^{d}} {16},\]
we have
\[\fint_{B_{r_0}} |u(x) - h(0)|^p \dd x \leq \frac 1 4.\]
Applying \eqref{e:average} with $\mu = h(0)$, the conclusion of the lemma follows, with $\eps>0$ determined from our choice of $\tau$ via Lemma \ref{l:rho}.
\end{proof}

\begin{lemma}\label{l:global-BMO}
There exist $\eps>0$, $r_0\in (0,\frac 1 4)$ and $C>0$, depending only on $d$, $\lambda$, and $\Lambda$, such that if $u$ is a minimizer of $\mathcal J_{A,\varphi}$ over $H^1_g(B_1)$ with $\|\varphi\|_{L^{d/2}_{\rm weak}(B_1)} \leq \eps$ and $\fint_{B_1} u^2 \dd x \leq 1$, then for any $r\in (0,r_0]$,
\[ \fint_{B_r} |u- u_r|^2 \dd x \leq C.\]
\end{lemma}

\begin{proof}
Let $\eps$ and $r_0$ be the constants from Lemma \ref{l:local-BMO}, and let $u$ be a minimizer of $\mathcal J_{A,\varphi}$ over $H^1_g(B_1)$. Our goal is to show by induction that
\begin{equation}\label{e:induction2}
 \fint_{B_{r_0^k}} |u- u_{r_0^k}|^2 \dd x \leq 1, \quad k = 1,2, \ldots.
\end{equation}

The base case $k=1$ follows directly from Lemma \ref{l:local-BMO}. Assuming that \eqref{e:induction2} holds for some $k\geq 1$, define
\[ v(x) =  u(r_0^k x) - u_{r_0^k} , \quad x\in B_1.\] 
By Lemma \ref{l:local-min}, $v$ is a minimizer of $\mathcal J_{A_k,\varphi_k,\gamma_k}$ over $H^1_{g_k}(B_1)$, where 
\[ \begin{split}
A_k(x) &= A(x_0+r_0x),\\
\varphi_k(x) &= r_0^{2k} \varphi(r_0^k x),\\
\gamma_k &=  -u_{r_0^k},\\
g_k(x) &=  u(r_0^k x) - u_{r_0^k}, \quad x\in \partial B_1.
\end{split}\]
The inductive hypothesis \eqref{e:induction2} implies 
\[\fint_{B_1}v^2 \dd x =  \fint_{B_{r_0^k}} |u-u_{r_0^k}|^2 \dd x \leq 1.\]
With the choice $q= d/2$, the scaling of $\|\varphi_k\|_{L^q_{\rm weak}}$ is as follows:
\[\|\varphi_k\|_{L^{d/2}_{\rm weak}(B_1)} = \|\varphi\|_{L^{d/2}_{\rm weak}(B_{r_0^k})} \leq \|\varphi\|_{L^{d/2}_{\rm weak}(B_1)} < \eps,\]
and we can apply Lemma \ref{l:local-BMO} to conclude 
\begin{equation}\label{e:v-estimate}
 \fint_{B_{r_0}} |v - v_{r_0}|^2 \dd x \leq 1 ,
\end{equation}
A quick calculation shows $v_{r_0} = u_{r_0^{k+1}} - u_{r_0^k}$. Therefore, \eqref{e:v-estimate} implies
\[ \fint_{B_{r_0^{k+1}}} |u -u_{r_0^{k+1}}|^2 \dd x \leq  1,\]
and we conclude \eqref{e:induction2} holds for all $k$. 

Now, for $0<r\leq r_0$, choose $k$ such that $r_0^{k+1} < r \leq r_0^k$. By \eqref{e:induction2},
\[ \fint_{B_r} |u - u_{r_0^k}|^2 \dd x \leq  \left(\frac {r_0^k} r \right)^d \fint_{B_{r_0^k}} |u-u_{r_0^k}|^2 \dd x \leq r_0^{-d}, \]
and the proof is complete. 
\end{proof}

Now we apply a scaled version of Lemma \ref{l:global-BMO} to conclude interior BMO regularity for $u$:

\begin{theorem}\label{t:BMO}
For any $\varphi \in L^{d/2}(\Omega)$ and $g \in H^{1/2}(\partial \Omega)$, and any minimizer $u$ of $\mathcal J_{A,\varphi}$ over $H^1_g(\Omega)$, there holds for any $\Omega' \subset \subset \Omega$,
\[ \| u\|_{BMO(\Omega')} \leq C \|u\|_{L^2(\Omega)},\]
where the constant $C$ depends only on universal constants, $\|\varphi\|_{L^{d/2}_{\rm weak}(\Omega)}$, and $\Omega'$.
\end{theorem}

\begin{proof}
Letting $\eps$ be the constant from Lemma \ref{l:global-BMO}, for any $x_0 \in \Omega$, we define
\[ w(x) := \kappa u(x_0+rx),\]
where 
\[ r = \min\left\{\frac 1 2 \mathrm{dist}(x_0,\partial B), 1\right\}, \quad \kappa^2 =  \min\left\{ \frac {1} {\fint_{B_r(x_0)} u^2 \dd x} ,   \frac \eps {\|\varphi(x_0 + rx)\|_{L^{d/2}_{\rm weak}(B_1)}}\right\}\]
Then, by Lemma \ref{l:local-min}, $w$ minimizes $\mathcal J_{\tilde A,\tilde \varphi}$ with $\tilde A(x) = A(x_0+rx)$, $\tilde \varphi(x) = r^2 \kappa^2 \varphi(x_0+rx)$, $\fint_{B_1} w^2 \leq 1$, and $\|\tilde \varphi\|_{L^{d/2}_{\rm weak}(B_1)} \leq \eps$. Applying Lemma \ref{l:global-BMO} and translating back to $u$, we conclude
\[ \fint_{B_r(x_0)} |u-u_{B_r(x_0)}|^2 \dd x \leq C\int_{\Omega} u^p \dd x,\]
for a constant $C$ depending only on universal constants and $\dist(x_0,\partial\Omega)$. 
\end{proof}

\section{Improved H\"older exponent}\label{s:improved}

For the one-phase problem, at free boundary points, we are able to improve the H\"older estimate to obtain the sharp exponent $\alpha = 1-d /(2q)$, regardless of the regularity theory available for $A$-harmonic functions.

First, we revisit the compactness argument of Lemma \ref{l:rho}, equipped with the extra regularity provided by Lemma \ref{l:global-holder}. 

\begin{lemma}\label{l:rho2}
Given $\tau>0$, there exists $\eps = \eps(d,\lambda,\Lambda,q,\tau) >0$ such that for any $g \in H^{1/2}(\partial B_1)$ with $g \geq 0$, any $A\in M_{\lambda,\Lambda}(B_1)$, any $\varphi$ such that $\|\varphi\|_{L^q_{\rm weak}(B_1)} < \eps$, and any minimizer $u$ of $\mathcal J_{A,\varphi}$ over $H_g^{1}(B_1)$ with $\fint_{B_1} u^2 \dd x \leq 1$ and $u(0) = 0$, there holds
\[\sup_{B_{1/10}} u \leq \tau.\]
\end{lemma}

\begin{proof}
Assume there is no $\eps$ satisfying the conclusion. Then, as in the proof of Lemma \ref{l:rho}, there exist $\eps_n\to 0$, $g_n\in H^{1/2}(\partial B_1)$, $A_n\in M_{\lambda,\Lambda}(B_1)$, and $\varphi_n$ with $\|\varphi_n\|_{L^q(B_1)} = \eps_n$, with minimizers $u_n$ of $\mathcal J_{A_n,\varphi_n}$ over $H^1_{g_n}(B_1)$, with $u_n(0) = 0$ and $\fint_{B_1} u_n^2 \dd x \leq 1$. Since $g_n\geq 0$, the maximum principle implies $u_n\geq 0$. The contradiction assumption implies
\begin{equation}\label{e:contradiction2}
 \sup_{B_{1/10}} u_n > \tau,
 \end{equation}
for all $n$.

Lemma \ref{l:caccioppoli} provides weak-$H^1(B_{1/2})$ compactness, so $u_n$ converge strongly in $L^2(B_{1/2})$ to a limit $u_0$. As in the proof of Lemma \ref{l:rho}, we see $u_0$ is a weak solution of \eqref{e:homogeneous} in $B_{1/2}$. In addition,  the $C^\alpha$ estimate of Lemma \ref{l:global-holder} implies $\|u_n\|_{C^\alpha(B_{1/2})}$ is bounded independently of $n$ for some $\alpha \in (0,1)$, which implies $u_n \to u_0$ uniformly in $B_{1/2}$ (up to a subsequence). For the limit $u_0$, this implies $u_0\geq 0$ in $B_{1/2}$ and $u_0(0) = 0$. Since the trace of $u_0$ on $\partial B_{1/2}$ is nonnegative, the strong maximum principle implies $u_0$ is identically $0$ in $B_{1/2}$. Uniform convergence implies that for $n$ large enough,
\[ \sup_{B_{1/10}} u_n \leq \tau,\]
contradicting \eqref{e:contradiction2}. 
\end{proof}

Now we iterate Lemma \ref{l:rho2} to prove H\"older regularity at boundary points, with the optimal exponent:

\begin{proof}[Proof of Theorem \ref{t:improved}]
Take $\tau = 1/10$, and let $\eps$ be the corresponding constant from Lemma \ref{l:rho2}. As in the proof of Theorem \ref{t:holder-prelim}, we choose $\kappa$ and $r$ such that, after replacing $u$ with
\[w(x):= \kappa u(x_0 + rx),\]
we can ensure $w(0) = 0$, $\fint_{B_1} w^2 \dd x \leq 1$, and that $w$ minimizes $\mathcal J_{A,\varphi}$ over $H^1_g(B_1)$ with $\|\varphi\|_{L^{q}_{\rm weak}(B_1)}\leq \eps$ and $g\geq 0$.

Applying Lemma \ref{l:rho2}, we have
\[ \sup_{B_{1/10}} w \leq \frac 1 {10}.\]
Next, with $\alpha = 1-d/(2q)$, assume by induction that for some positive integer $k$,
\begin{equation}\label{e:induction3}
 \sup_{B_{10^{-k}}} w \leq \frac 1 {10^{k\alpha}}.
 \end{equation}
By Lemma \ref{l:local-min}, the function 
\[w_k(x) := 10^{k\alpha} w\left(\frac x {10^k}\right)\]
minimizes $\mathcal J_{A_k,\varphi_k}$ over $H^1_{g_k}(B_1)$ with the same value of $\rho$, and 
\[ \begin{split}
\varphi_k(x) &= 10^{2k(1-\alpha)}\varphi(10^{-k}x),\\
A_k(x) &= A(10^{-k} x),\\
g_k(x) &= 10^{k\alpha} w(10^{-k} x), \quad x\in \partial B_1.
\end{split}\]
The inductive hypothesis gives
\[\fint_{B_1} w_k^2 \dd x = 10^{2k\alpha} \fint_{B_{10^{-k}}} w^2 \dd x \leq 1,\]
and our choice of $\alpha = 1-d/(2q)$ gives
\[ \|\varphi_k\|_{L^q_{\rm weak}(B_1)} =  10^{2k(1-\alpha)- kd/q} \|\varphi\|_{L^q_{\rm weak}(B_{10^{-k}})} \leq \eps.       \]
Lemma \ref{l:rho2} now implies $\sup_{B_{1/10}} w_k \leq 1/10$, or
\[ \sup_{B_{10^{-(k+1)}}} w\leq \frac 1 {10^{(k+1)\alpha}},\]
and we have established that \eqref{e:induction3} holds for all $k=1,2,\cdots$. 

Now, for any $x \in B_{1/10}$, there is some $k$ such that $10^{-(k+1)} < |x| \leq 10^{-k}$. By \eqref{e:induction3}, there holds
\[ |w(x)| \leq \sup_{B_{10^{-k}}} \leq \frac 1 {10^{k\alpha}} \leq 10^\alpha|x|^{\alpha}.\]
After translating from $w$ to $u$, the proof is complete. 
\end{proof}

\section{Geometry of the free boundary} \label{sct geometry}

The following lemma is a nondegeneracy estimate for $u$ near free boundary points. 

\begin{lemma}\label{l:nondegeneracy}
Let $u$ be a minimizer of $\mathcal J_{A, \varphi}$ over $H^1_g(\Omega)$ with $g$ and $u$ nonnegative. 
Let $x_0 \in \partial \{u>0\}$ be a free boundary point, and for some $p>1$ and $\sigma \in [0,d)$, assume that there exist $c_0,r_0>0$ such that
\begin{equation}\label{e:growth}
{|\{\varphi(x) > t\}\cap B_{r}(x_0)| \geq \min\left(c_0 r^\sigma t^{-p}, |B_r(x_0)|\right), \quad \text{for all } t>0, r\in (0,r_0).}
\end{equation}
Then, for any $x\in \Omega$ with $u(x) > 0$ such that $r:= |x-x_0| = \dist(x,\partial\{u>0\}) \leq r_0/2$, the estimate
\[
{ u(x) \geq C |x-x_0|^{1-(d-\sigma)/(2p)}.}
\]
holds, with $C>0$ depending only on $d$, $p$, $\lambda$, $\Lambda$, and $c_0$.
\end{lemma}

{ In \eqref{e:growth}, the case $\sigma = 0$ corresponds, for example, to an inverse power function $\varphi(x) = |x-x_0|^{-d/p} \in L^p_{\rm weak}(\Omega)$. In such a case, Lemma \ref{l:nondegeneracy} gives a lower bound for $u$ with exponent $1-d/(2p)$, which matches the asymptotics of the H\"older estimate of Theorem \ref{t:improved}. The cases with $\sigma > 0$ include choices such as $\varphi(x) = \dist(x,\Gamma)^{(m-d)/p}$ for a smooth $m$-dimensional submanifold $\Gamma$, with $1\leq m\leq d-1$. This example is also in $L^p_{\rm weak}(\Omega)$.}

\begin{proof}[Proof of Lemma \ref{l:nondegeneracy}]
Let $x = x_\eps$ satisfy $u(x_\eps) = \eps>0$. 
Our goal is to bound $\eps$ from below in terms of the proper power of $r$. 

Since $u$ satisfies the homogeneous equation \eqref{e:homogeneous} in $B_{3r/4}(x_\eps) \subset \{u>0\}$, the Harnack inequality, see for instance \cite[Theorem 5]{serrin1964}, gives a constant $C_0>0$, independent of $r$, with 
\[ 
u(x) \leq C_0 \eps, \quad x\in \partial B_{r/2}(x_\eps).
\]
Let $\zeta$ be a smooth cutoff function equal to 0 in $B_{r/4}(x_\eps)$ and equal to 1 outside $B_{r/2}(x_\eps)$, with $|\nabla \zeta|\leq C r^{-1}$. Define the test function 
\[ 
v(x) := \min\{ u(x), C_0\eps \zeta(x)\}.
\]
By construction, $v = u$ on $\partial B_{r/2}(x_\eps)$, so the minimizing property $\mathcal J_{A,\varphi}(u) \leq \mathcal J_{A,\varphi}(v)$ implies
\begin{equation}\label{e:u-v-compare}
\begin{split}
 \int_{B_{r/2}(x_\eps)} \varphi(x)\left(1_{\{u>0\}}- 1_{\{v>0\}}\right) \dd x &\leq \frac 1 2 \int_{B_{r/2}(x_\eps)} \left(  \nabla v\cdot (A(x)\nabla v) - \nabla u\cdot(A(x)\nabla u)\right) \dd x\\
 &\leq \frac 1 2 \Lambda C_0^2 \eps^2 \int_{B_{r/2}(x_\eps)\cap \{v> u\}} |\nabla \zeta|^2 \dd x\\
 &\leq \frac 1 2 \Lambda C_0^2 \eps^2 \omega_d (r/2)^{d} C r^{-2}.
 \end{split}
 \end{equation}
The left-hand side of this inequality can be bounded from below as follows, using our choice of $\zeta$:
\begin{equation}\label{e:phi-int-ineq}
 \int_{B_{r/2}(x_\eps)} \varphi(x)\left(1_{\{u>0\}}- 1_{\{v>0\}}\right) \dd x \geq \int_{B_{r/2}(x_\eps)\setminus B_{r/4}(x_\eps)} \varphi(x) \dd x.
 \end{equation}
To bound the last expression from below, we choose a suitable $t>0$ in the blow-up condition \eqref{e:growth} so that $\varphi>t$ in a large percentage of $B_{2r}(x_0)$, which must include at least half of $B_{r/2}(x_\eps)\setminus B_{r/4}(x_\eps)$. In more detail, 
let 
\[
\begin{split}
\mu_d &= \frac{|B_{r/2}(x_\eps)\setminus B_{r/4}(x_\eps)|} {|B_{2r}(x_0)|} \in (0,1),\\
k &= \left[\frac 1 {c_0} 2^{d} \omega_d(1-\mu_d/2)\right]^{-1/p},
\end{split}
\] 
where $\omega_d = |B_1|$.  Then, since $2r\leq r_0$, \eqref{e:growth} with $t = k r^{(\sigma-d)/p}$ implies 
\[
|\{\varphi(x)>k r^{(\sigma-d)/p}\}\cap B_{2r}(x_0)| \geq c_0 k^{-p} r^{d} = (1-\mu_d/2) |B_{2r}(x_0)|.
\] 
Therefore, 
\[
|\{\varphi(x) > k r^{(\sigma-d)/p}\}\cap (B_{r/2}(x_\eps)\setminus B_{r/4}(x_\eps))| \geq \frac 1 2  |B_{r/2}(x_\eps)\setminus B_{r/4}(x_\eps)|,
\] 
and 
\[ 
\int_{B_{r/2}(x_\eps)\setminus B_{r/4}(x_\eps)} \varphi(x) \dd x \geq \frac 1 2 |B_{r/2}(x_\eps)\setminus B_{r/4}(x_\eps)| k r^{(\sigma-d)/p} = c_d k r^{d(1-1/p)+\sigma/q}.
 \]
Note that $k$ depends only on $d$, $p$, and $c_0$. Combining this with \eqref{e:u-v-compare}, we finally have
\[ 
\eps^2 \geq C r^{2 -(d-\sigma)/p},
\]
for a constant $C>0$ as in the statement of the lemma, and the proof is complete.
\end{proof}

{Next, we have a generalization of Lemma \ref{l:nondegeneracy} for Bernoulli functions $\varphi$ that are singular as $x\to x_0$ only from certain directions. A typical example would be a $\varphi$ that blows up on one side of a hypersurface but is bounded on the other side, such as $\varphi(x) = 1 + 1_{\{x_1>0\}}x_1^{-1/p}$. We are mainly interested in this generalization so that we can rigorously prove our example in Section \ref{sct example} has a free boundary point where $|\nabla u|$ is infinite.

\begin{lemma}\label{l:nondegeneracy2}
Let $u$ be a minimizer of $\mathcal J_{A, \varphi}$ over $H^1_g(\Omega)$ with $g$ and $u$ nonnegative. Let $x_0 \in \partial \{u>0\}$ be a free boundary point, and for some $p>1$, $\sigma \in [0,d)$, and some cone $\Xi\subset \R^d$ with vertex at $x_0$, assume that there exist $c_0,r_0>0$ such that
\begin{equation}\label{e:growth2}
|\{\varphi(x) > t\}\cap B_{r}(x_0)\cap \Xi| \geq \min\left(c_0 r^\sigma t^{-p}, |B_r(x_0)\cap \Xi|\right), \quad \text{for all } t>0, r\in (0,r_0).
\end{equation}
Then, for any $x\in \Omega\cap \Xi$ with $u(x) > 0$ such that $r:= |x-x_0| = \dist(x,\partial\{u>0\}) \leq r_0/2$, the estimate
\[
u(x) \geq C |x-x_0|^{1-(d-\sigma)/(2p)}.
\]
holds, with $C>0$ depending only on $d$, $p$, $\lambda$, $\Lambda$, $c_0$, and $\Xi$.
\end{lemma}

\begin{proof}
This lemma is proven by the same method as Lemma \ref{l:nondegeneracy}, with the following alteration: to obtain a lower bound for the integral on the right in \eqref{e:phi-int-ineq}, one applies the condition \eqref{e:growth2} with $t = kr^{(\sigma-d)/p}$ and $k$ chosen depending on $\Xi$ so that 
\[
c_0k^{-p}r^d = \left | B_{2r}(x_0)\cap \Xi \right | - \frac 1 2 \left | (B_{r/2}(x_\eps)\setminus B_{r/4}(x_\eps)) \cap \Xi \right |.
\]
This implies 
$$
	\left |\{\varphi(x) > kr^{(\sigma- d)/p}\} \cap (B_{r/2}(x_\eps) \setminus B_{r/4}(x_\eps))\cap \Xi \right | \geq \frac 1 2 \left |(B_{r/2}(x_\eps)\setminus B_{r/4}(x_\eps)) \cap \Xi \right |,
$$ 
and the remainder of the argument proceeds as in the proof of Lemma \ref{l:nondegeneracy}
\end{proof}
}

Now we are ready to prove our last main result, which allows us to control the severity of cusps along the free boundary:

\begin{proof}[Proof of Theorem \ref{t:cusps}]

For a free boundary point $x_0$, let $q_+$ and $q_-$ be as in the statement of the Theorem, i.e. $\varphi \in L^{q_+}(B_{r_0}(x_0))$ and for $r\in (0,r_0)$,
\[
|\{\varphi(x)>t\} \cap B_{r}(x_0)| \geq { \min\left( c_0 t^{-q_-}, |B_r(x_0)|\right)},
\]
Then Theorem \ref{t:improved} implies $u(x) \leq C|x-x_0|^{1-d/(2q_+)}$ in $B_{r}(x_0)$, and Lemma \ref{l:nondegeneracy} implies $u(x) \geq c_{q_-} |x-x_0|^{1-d/(2q_-)}$ in $B_{r}(x_0)$. 

To keep the notation brief, we define $\alpha_1 = 1-d/(2q_+)$ and $\alpha_2 = 1-d/(2q_-)$. 

For fixed $r\in (0,r_0)$, let $x_1\in \partial B_r(x_0)$ be such that $u(x_1) \geq \dfrac {c_{q_-}} 2 r^{\alpha_2}$. By Theorem \ref{t:main-holder}, $u$ is H\"older continuous of order $\alpha_1$ at $x_1$. Define $\rho = r^{\alpha_2/\alpha_1}$. For $\kappa>0$ and $y\in B_{\kappa \rho}(x_1)$, we have
\[   u(y) \geq \frac {c_{q_-}} 2 r^{\alpha_2}  - C (\kappa \rho)^{\alpha_1} = \left( \frac {c_{q_-}} 2 - C\kappa ^{\alpha_1} \right) r^{\alpha_2},\]
which is strictly positive if we choose $\kappa = [c_{q_-}/(4C)] ^{1/\alpha_1}$. We conclude $B_{\kappa \rho}(x_1) \subset \{u>0\}$. A simple geometric argument gives $|B_r(x_0) \cap B_{\kappa \rho}(x_1)| \geq c(\kappa) \rho^d$, which implies
\begin{equation}\label{e:alpha1alpha2}
 \frac{|B_r(x_0) \cap \{u>0\}|}{ r^{d \alpha_2/\alpha_1}} \geq c,
 \end{equation}
for a constant $c>0$ depending on $q_-$ and $q_+$ but independent of $r$.

In the case $q_{+} = q_{-}$, we can in fact take $\alpha_1 = \alpha_2$ in \eqref{e:alpha1alpha2}, yielding a corner-like estimate.

For the lower bound on $|B_r(x_0)\cap \{u=0\}|$, arguing as in \cite{ACF1984quasilinear}, we compare $u$ to the function $v$ defined by
\[ \begin{cases}\nabla \cdot (A(x) \nabla v)= 0 &\text{ in } B_r(x_0),\\
v-u \in H_0^{1}(B_r(x_0)). &\end{cases}\]
Since $v-u\geq 0$ on $\partial B_r(x_0)$, the maximum principle and Lemma \ref{l:subsolution} imply $v\geq u$ in the interior of $B_r(x_0)$.

Using the Poincar\'e inequality, the identity $\int_{B_r(x_0)} \nabla (u-v)\cdot (A(x) \nabla v) \dd x = 0$, and the minimizing property of $u$, we have 
\begin{equation}\label{e:u-v}
\begin{split}
\frac 1 {r^2} \int_{B_r(x_0)} |u-v|^2 \dd x &\leq C\int_{B_r(x_0)} |\nabla (u-v)|^2 \dd x\\
&\leq C\lambda^{-1} \int_{B_r(x_0)} [\nabla u\cdot (A(x) \nabla u) - \nabla v\cdot (A(x) \nabla v)]\dd x\\
 &\leq C\lambda^{-1}\int_{B_r(x_0)} \varphi(x) 1_{\{u=0\}} \dd x\\
&\leq C \lambda^{-1}\|\varphi\|_{L^{q_+}(B_r(x_0))} |B_r(x_0)\cap \{u=0\}|^{1-1/q_+}.
\end{split}
\end{equation}

As above, let $\rho = r^{\alpha_2/\alpha_1}$. The Harnack inequality (applied to $v$) and the nondegeneracy of $u$ (Lemma \ref{l:nondegeneracy}) imply, for $\kappa\in (0,1)$ sufficiently small,
\[  v(x) \geq c \sup_{B_{r/2}(x_0)} u \geq c r^{\alpha_2}, \quad x\in B_{\kappa \rho(x_0)}. \] 
Since $u(x_0) = 0$, the H\"older continuity of $u$ implies $u\leq C (\kappa \rho)^{\alpha_1}$ in $B_{\kappa \rho}(x_0)$. We therefore have
\[ v(x) - u(x) \geq c r^{\alpha_2} - C(\kappa \rho)^{\alpha_1} = (c-C\kappa^{\alpha_1}) r^{\alpha_2} \geq \frac c 2 r^{\alpha_2},\]
if we choose $\kappa = [c/(2C)]^{1/\alpha_1}$. With \eqref{e:u-v}, we now have
\[ |B_r(x_0) \cap \{u=0\}|^{1-1/q_+} \geq K r^{-2+ 2\alpha_2} \rho^d = K r^{2(\alpha_2-1) + d \alpha_2/\alpha_1},\]
or
\[ |B_r(x_0) \cap \{u=0\}| \geq  K r^{d(\alpha_2/\alpha_1 - 1/q_-)q_+/(q_+-1)}.\]

Finally, in the case $q_{+} = q_{-}$,  we obtain $|B_r(x_0)\cap \{u=0\}| \geq K r^d$, as in the classical theory.
\end{proof}

\section{An example of a free boundary point with infinite jump} \label{sct example}

In this final section, we discuss an example that shows the free boundary can indeed intersect the infinite set of $\varphi$. 

For $d\geq 3$ fixed, let $m>0$ be a constant to be chosen later, and define
\[ 
\tau(r) := \frac {m(d-2)}{r - r^{d-1}}, \quad r\in (0,1).
\]
The function $\tau$ arises in the analysis of radially symmetric harmonic functions, which are involved in our proof below. The minimum of $\tau$ is achieved at 
\[
 r_* = \left(\frac 1 {d-1}\right)^{\frac 1 {d-2}} \in (0,1),
 \]
 and $\tau(r_*) = \left(m r_*^{d-1}\right)^2$. Next, let $r^* = \dfrac {1+r_*} 2$, and for fixed $q>1$, define $\varphi: B_1 \to (0,+\infty]$ by
\[ 
\varphi(x) = 
\begin{cases} 
\left(m r_*^{d-1}\right)^2 , & |x| <r_* ,\\ 
\left( |x|-r_*\right)^{-1/q}, & r_* \leq |x|< r^*,\\
\left(m r_*^{d-1}\right)^2, &r^*\leq |x|\leq 1,
\end{cases}
\]
and note that $\varphi \in L^{q}_{\text weak}(B_1)$. We consider the minimization problem for 
\[
 \mathcal J_{I,\varphi}(v) = \int_{B_1} (|\nabla v|^2 + \varphi(x) 1_{\{v>0\}}) \dd x, \quad v\in H^1_g(B_1),
\]
with $g\equiv m$ on $\partial B_1$ for some constant $m$. We claim that for $m>0$ sufficiently small, depending only on $d$ and $q$, there is a minimizer $u$ such that $\partial\{u>0\}$ intersects $\{\varphi(x) = \infty\}$.
 
 {Since $\varphi$ is not a monotonic function of $|x|$, we cannot apply rearrangement methods to conclude minimizers of $\mathcal J_{I,\varphi}$ are radially symmetric. In fact, we will show that a non-symmetric minimizer exists for certain choices of $m$. }
 
 { On a technical note, our argument below assumes $u$ is differentiable at free boundary points where $\varphi < \infty$. This is not always true in the pointwise sense, but by understanding solutions to the Bernoulli problem in the viscosity sense, our argument (which uses nothing more than the comparison principle) can be made rigorous (see, {for instance}, \cite[Section 6]{MT} for a detailed discussion of the meaning of the free boundary condition $|\nabla u|^2 = \varphi$ in the viscosity sense). We omit the details about this issue because it concerns the boundary condition at points where $\varphi$ is finite, which is suitably explained by existing theory.}

First, we define a useful class of comparison functions: radially symmetric functions that are zero in $B_{r}$ for some $r\in (0,1)$, and harmonic in $B_1\setminus B_{r}$, with boundary values equal to $m$ on $\partial B_1$. Explicitly, these functions are given by
\begin{equation}\label{e:r}
u_r(x) = 
\begin{cases}
0, & |x|< r,\\
\frac m {r^{2-d}-1} (r^{2-d} - |x|^{2-d}), & |x|\geq r,
\end{cases}
\end{equation}
and they have energy
\[
\begin{split}
\mathcal J_{I,\varphi}(u_r) &=  \omega_d\int_{r}^1 \left(\frac{ m^2 (d-2)^2}{(r^{2-d} - 1)^2} \rho^{2-2d} + \varphi(\rho)\right) \rho^{d-1} \dd \rho\\
&= \omega_d\left(\frac {m^2 (d-2)}{r^{2-d}-1} + \int_r^1 \varphi(\rho)\rho^{d-1} \dd \rho\right),
\end{split}
\]
where $\omega_d$ is the measure of $\mathbb S^{d-1}$, and we have written $\varphi(\rho) = \varphi(|x|)$. For any $r\in (0,1)$, the function $u_r$ is admissible for the minimization problem.

Now, let $u$ be a minimizer. We claim that $\{u=0\}$ cannot be empty, if $m$ is chosen sufficiently small. Indeed, if $u$ is positive in all of $B_1$, then it is harmonic in $B_1$ and must be identically equal to $m$. Then $\mathcal J_{I,\varphi}(u) = \int_{B_1} \varphi \dd x$.  To rule out this case, we would like to find $r\in (0,1)$ such that $u_r$ defined by \eqref{e:r} has energy less than $\int_{B_1}\varphi \dd x$, or in other words,
\begin{equation}\label{e:m-ineq}
 m^2 \frac {d-2}{r^{2-d}- 1} < \int_0^{r} \varphi(\rho) \rho^{d-1} \dd \rho.
\end{equation}
Choosing $r = r^*$, we see that 
\[
\int_0^{r^*} \varphi(\rho) \rho^{d-1} \dd \rho \geq \int_{r_*}^{r^*} \varphi(\rho) \rho^{d-1} \dd \rho = \int_{r_*}^{r^*} (\rho - r_*)^{-1/q} \rho^{d-1} \dd \rho.
\]
Therefore, by choosing $m>0$ small enough that
\[ 
m^2 < \frac{(r^*)^{2-d} - 1}{d-2} \int_{r_*}^{r^*} (\rho-r_*)^{-1/q} \rho^{d-1} \dd \rho,
\]
we ensure \eqref{e:m-ineq} is satisfied when $r=r^*$,  so $u\equiv m$ cannot be a minimizer. Note that such $m$ can be chosen depending only on $d$ and $q$.

Now, since $\{u=0\}$ is not empty, let $x_1$ and $x_2$ be the points on $\partial \{u=0\}$ of smallest and largest magnitude, respectively. Define $r_1 = |x_1|$ and $r_2 = |x_2|$. For $i=1,2$, let $u_i:=u_{r_i}$ denote the function defined in \eqref{e:r} with $r=r_i$. 

By our choice of $x_1$ and $x_2$, we have $u_1\geq u \geq u_2$ on $\partial \{u>0\}$. The comparison principle implies $u_1 \geq u \geq u_2$ in all of $B_1$, and 
\begin{equation}\label{e:r1r2}
\partial_r u_1(x_1) \geq \partial_r u(x_1), \quad \partial_r u_2(x_2) \leq \partial_r u(x_2),
\end{equation}
where $\partial_r$ is the derivative in the radial direction. These two inequalities will imply useful bounds on $r_1$ and $r_2$.

Starting with $r_2$, the definition of $u_2$ implies, with \eqref{e:r1r2},
\[
\partial_r u(x_2) \geq \frac{(d-2) m} { r_2 - r_2^{d-1}} = \tau(r_2).
\]
On the other hand, since $u$ is a minimizer, we have 
$\partial_r u(x_2) = \sqrt{\varphi(x_2)}$,  so that 
\begin{equation}\label{e:r2-ineq}
\sqrt{\varphi(x_2)} \geq \tau(r_2).
\end{equation}
This implies $r_2 = |x_2|$ must lie in the part of $[0,1]$ where $\sqrt\varphi \geq \tau$. Choosing $m>0$ smaller if necessary (depending only on $d$ and $q$) we can ensure that 
\[
\lim_{r\to r^*-} \varphi(r) = (r^*-r_*)^{-1/q} > \frac{(d-2)m}{r^* - (r^*)^{d-1}} = \tau(r^*).
\]
Since $\tau$ is increasing on $(r_*,1)$, we also have 
\[\tau(r^*) >  \tau(r_*) = \lim_{r\to r^*+} \varphi(r).\]
Therefore, the inequality $\sqrt{\varphi(x_2)} \geq \tau(r_2)$ implies $r_2 \in [r_*,r^*]$.

Regarding $r_1$, we similarly have from \eqref{e:r1r2} that
\begin{equation}\label{e:r2eqn}
\tau(r_1) =  \frac {(d-2) m} {r_1 - r_1^{d-1}} \geq \partial_r u(x_1) = \sqrt{\varphi(x_1)}
\end{equation}
Since $\sqrt{\varphi(r)} > \tau(r)$ for $r\in (r_*,r^*)$, inequality \eqref{e:r2eqn} implies $r_1 \in [0,r_*]\cup[r_*,1]$. Since $r_1\leq r_2$ by definition and $r_2 \in [r_*,r^*]$, we in fact have $r_1\in [0,r_*]\cup\{r^*\}$. 

Next, we would like to rule out the case $r_1 = r_2 = r^*$. In this case, the inequalities $u_1 \geq u\geq u_2$ imply $u_1 = u = u_2$, and therefore $u$ is given by \eqref{e:r} with $r = r^*$.  In particular, $\partial_r u(x_2) = \tau(r^*)$. But in the set $\{u>0\} = B_1\setminus B_{r^*}$, there holds $\varphi \equiv \left(m r_*^{d-1}\right)^2$, so one should have $\partial_r u(x_2) = mr_*^{d-1} = \tau(r_*)< \tau(r^*)$ for any solution of the Bernoulli problem, which is a contradiction. We conclude 
\[
r_1 \leq r_*\leq r_2.
\]

Since  $u_1 \geq u\geq u_2$, we clearly have $B_{r_1} \subset \{u=0\} \subset B_{r_2}$, and with $r_1 \leq r_* \leq r_2$, this implies the existence of at least one point $x_0\in\partial\{u>0\}$ with $|x_0|=r_*$, and such that $\{\varphi(x) = +\infty\} = \partial B_{r_*}$ intersects $\{u>0\}\cap B_\rho(x_0)$ for any $\rho>0$, as claimed. 

{ For any $x_0 \in \partial B_{r_*}$, there is a cone $\Xi$ with vertex at $x_0$ and aperture $\pi$ (i.e. $\Xi$ is a half-plane), such that 
\[
|\{\varphi(x) > t\}\cap B_r(x_0)\cap \Xi| \geq \min\left(c r^{d-1} t^{-q}, |B_r(x_0)\cap \Xi|\right),
\] 
for $r>0$ sufficiently small. 
From Lemma \ref{l:nondegeneracy2} with $p=q$ and $\sigma = d-1$, we have $u(x) \geq C|x-x_0|^{1-1/(2q)}$ for $x\in \Xi$ near $x_0$. We conclude $|\nabla u(x)|\to \infty$ as $x$ approaches $x_0$ from inside $\Xi \cap \{u>0\}$. }

\bibliographystyle{amsplain, amsalpha}

\end{document}